\documentclass[10pt,reqno]{amsart}
\usepackage{amssymb}

\usepackage{float}
\usepackage{graphicx}

\usepackage{enumerate}
\usepackage{float}
\usepackage{ifthen}

\usepackage{centernot}
\usepackage{mathtools}
\usepackage{ stmaryrd }

\usepackage{mathtools}

\usepackage{graphicx}

\nonstopmode\numberwithin{equation}{section}

\usepackage[applemac]{inputenc}

\newtheorem{definition}{Definition}[section]
\newtheorem{theorem}{Theorem}[section]

\newtheorem{lemma}{Lemma}[section]

\newtheorem{remark}{Remark}[section]

\newcommand{\be}{\begin{equation}}
\newcommand{\ee}{\end{equation}}

\allowdisplaybreaks
    
\begin{document}
\title{Velocity Syzygies and Bounding Syzygy Moments in the Planar Three-Body Problem }

\author{
Alexei Tsygvintsev \\
}
{
\noindent 
UMPA, Ecole Normale Sup\'{e}rieure de Lyon\\
Email: alexei.tsygvintsev@ens-lyon.fr
}

\bigskip
\begin{abstract}
We consider the Newtonian planar three-body problem, defining a syzygy (velocity syzygy) as a configuration where the positions (velocities) of the three bodies become collinear. We demonstrate that if the total energy is negative, every collision-free solution has an infinite number of velocity syzygies. Specifically, the velocities of the three bodies become parallel within every interval of time containing three consecutive syzygies. Using comparison theory for matrix Riccati equations, we derive new upper and lower bounds on the moments when syzygies occur.
 \end{abstract}

\maketitle

\section{Introduction and Preliminary Results}

The Newtonian planar three-body problem has been a subject of extensive study in celestial mechanics. It involves predicting the motion of three bodies moving under their mutual gravitational attraction. Understanding syzygies, where the positions or velocities of the three bodies become collinear, provides significant insight into the dynamics of the system.

Let \( P_1 \), \( P_2 \), and \( P_3 \) be three points in the plane with strictly positive masses \( m_1, m_2, m_3 \) and Euclidean coordinates \( (x_i, y_i) \in \mathbb{R}^2 \), \( i = 1, 2, 3 \). The Newtonian three-body problem \cite{W} can be formulated as follows:

\begin{equation} \label{PLC}
\ddot{z}_1 = m_2 \frac{z_{21}}{|z_{21}|^3} - m_3 \frac{z_{13}}{|z_{13}|^3}, \quad \ddot{z}_2 = m_3 \frac{z_{32}}{|z_{32}|^3} - m_1 \frac{z_{21}}{|z_{21}|^3}, \quad \ddot{z}_3 = m_1 \frac{z_{13}}{|z_{13}|^3} - m_2 \frac{z_{32}}{|z_{32}|^3}\,,
\end{equation}
where \( z_k = x_k + i y_k \in \mathbb{C} \), \( k = 1, 2, 3 \) and \( z_{kl} = z_k - z_l \). 

Given that the total linear momentum (which is conserved) is zero, one can always set
\begin{equation} \label{sums}
\sum\limits_k m_k \dot{z}_k = 0, \quad \sum\limits_k m_k z_k = 0\,,
\end{equation}
by placing the center of mass of the bodies at the origin.

Let \( t \mapsto z_i(t) \), \( i = 1, 2, 3 \) be any collision-free solution of equations \eqref{PLC} defined for \( t \in I = [0, a) \), \( a > 0 \), and determined by a set of initial conditions \( (z_i(0), \dot{z}_i(0)) \), \( i = 1, 2, 3 \).

\begin{definition} \label{D}
We say that the three bodies \( P_1, P_2, P_3 \) form a syzygy (velocity syzygy) at the moment \( t_0 \in I \) if the complex triplet \( (z_1, z_2, z_3)(t_0) \) (\( (\dot{z}_1, \dot{z}_2, \dot{z}_3)(t_0) \)) is collinear.
\end{definition}

Our first result concerns the existence of velocity syzygies.

\begin{theorem} \label{T2}
Let \( t \mapsto (z_1(t), z_2(t), z_3(t)) \), \( t \in [0, +\infty) \), be a zero angular momentum, collision-free solution to the three-body problem \eqref{PLC} with negative energy. Then, it has an infinite number of velocity syzygies.
\end{theorem}

\begin{proof}
After introducing the new variables \( w_i = m_i z_i \), \( i = 1, 2, 3 \), the relations \eqref{sums} yield
\begin{equation} \label{bnb}
\sum\limits_i w_i = \sum\limits_i \dot{w}_i = 0.
\end{equation}

Writing
\begin{equation} \label{1234}
w_k = X_k + i Y_k, \quad X_k = m_k x_k, \quad Y_k = m_k y_k, \quad k = 1, 2, 3\,,
\end{equation}
and using \eqref{bnb}, one derives from the equations \eqref{PLC} the following \( 2 \times 2 \) matrix equation:
\begin{equation} \label{EQ}
\ddot{X} = AX, \quad X =
\begin{pmatrix}
X_1 & Y_1 \\
X_2 & Y_2
\end{pmatrix}, \quad
A = \begin{pmatrix}
-m_2 \rho_3 - m_{13} \rho_2 & m_1 \rho_{32} \\
m_2 \rho_{31} & -m_1 \rho_3 - m_{32} \rho_1
\end{pmatrix}\,,
\end{equation}
where
\begin{equation}
\rho_1 = \frac{1}{|z_{32}|^3}, \quad \rho_2 = \frac{1}{|z_{13}|^3}, \quad \rho_3 = \frac{1}{|z_{21}|^3}, \quad m_{ij} = m_i + m_j, \quad \rho_{ij} = \rho_i - \rho_j\,.
\end{equation}

We define the determinants
\begin{equation}
\Delta_1(t) =
\left|
\begin{array}{ll}
X_1 & Y_1 \\
X_2 & Y_2
\end{array}
\right|(t), \quad \Delta_2(t) =
\left|
\begin{array}{ll}
\dot{X}_1 & \dot{Y}_1 \\
\dot{X}_2 & \dot{Y}_2
\end{array}
\right|(t), \quad t \in [0, +\infty)\,.
\end{equation}

It is sufficient to show that \( \Delta_2 \) has an infinite number of zeros for \( t \geq 0 \). As demonstrated by Montgomery \cite{M1}, and later independently by the author in \cite{TS1, TS2}, the solution \( t \mapsto (z_1(t), z_2(t), z_3(t)) \), for \( t \in [0, +\infty) \), has an infinite number of consecutive syzygies. That is, there exist \( t_i \geq 0 \), \( i = 1, 2, \dots \), with \( t_{i+1} > t_i \) such that \( \Delta_1(t_i) = 0 \). As shown in \cite[p. 6833]{TS2}, in the zero angular momentum case:
\begin{equation} \label{deltas}
\dot{\Delta}_1(t)^2 - 4 \Delta_1(t) \Delta_2(t) \geq 0, \quad \forall t \geq 0\,,
\end{equation}
where equality occurs only if the positions \( z_i(t) \) and velocities \( \dot{z}_i(t) \) for \( i = 1, 2, 3 \) are parallel.  

Let \( t_i < t_{i+1} < t_{i+2} \) be three consecutive syzygy moments. By Rolle's theorem, there exist \( \tau \in (t_i, t_{i+1}) \) and \( \eta \in (t_{i+1}, t_{i+2}) \) such that \( \tau < \eta \) and \( \dot{\Delta}_1(\tau) = \dot{\Delta}_1(\eta) = 0 \). Without loss of generality, we can assume that \( \Delta_1(\tau) > 0 \) and \( \Delta_1(\eta) < 0 \). Indeed, if for some \( i \geq 0 \), \(\Delta_1(t_i) = \dot{\Delta}_1(t_i) = 0\), then according to \eqref{deltas}, the corresponding solution is a straight-line one (with all positions and velocities lying on the same line) and will result in a triple collision since the energy is negative.

 Thus, by the Intermediate Value Theorem and \eqref{deltas}, \( \Delta_2 \) has a zero in the interval \( [\tau, \eta] \). The proof is complete.
\end{proof}

\section{Bounds on Syzygy Moments in the Zero Angular Momentum Case}

It is well-established that in the case of negative energy and zero angular momentum, every collision-free solution in the interval \([0, +\infty)\) has an infinite number of syzygies \cite{M1,TS1,TS2}. In the author's previous works \cite{TS1, TS2}, an upper bound was determined for the moment when the very first syzygy occurs, expressed as a function of the energy value and the masses alone. In this section, we refine our results by providing both upper and lower bounds that depend on the initial positions and velocities of the bodies, resulting in significantly more precise estimates.

First, we present a preliminary result from the comparison theory for matrix Riccati equations, developed by Eschenburg and Heintze in 1990 \cite{EH}, which will be utilised later in the proof of our main result.

Let \( E \) be a finite-dimensional real vector space equipped with an inner product \((\cdot, \cdot)\), and let \( S(E) \) denote the space of self-adjoint linear endomorphisms of \( E \). Consider the Riccati differential equation with a smooth coefficient curve \( R : (0, t_0) \to S(E) \):
\begin{equation} \label{riccati}
\dot{B} + B^2 + R = 0,
\end{equation}
where \( B : (0, t_0) \to S(E)\) is the solution.

Given two smooth coefficient curves \( R_1, R_2 : (0, t_0) \to S(E) \) with \( R_1 \geq R_2 \) (i.e., \( R_1 - R_2 \) is positive semidefinite), we can compare the solutions \( B_1 \) and \( B_2 \) of the Riccati equations:
\begin{equation}
\dot{B}_1 + B_1^2 + R_1 = 0,
\end{equation}
and
\begin{equation}
\dot{B}_2 + B_2^2 + R_2 = 0,
\end{equation}
subject to appropriate initial conditions.

\begin{theorem}[\cite{EH}] \label{EH1}
Let \( R_1, R_2 : (0, t_0) \to S(E) \) be smooth coefficient curves with \( R_1 \geq R_2 \). For \( j = 1, 2 \), let \( B_j : (0, t_j) \to S(E) \) be the solution to the Riccati equation corresponding to \( R_j \), with maximal \( t_j \in (0, +\infty] \). If the difference \( U := B_2 - B_1 \) has a continuous extension to \( 0 \) with \( U(0) \geq 0 \), then \( t_1 \leq t_2 \) and \( B_1 \leq B_2 \) on \( (0, t_1) \).
\end{theorem}

We now assume that at the instant \( t = 0 \), the configuration of the three-body problem is different from a syzygy, i.e., the positions of the three bodies are not parallel. In this case, the matrix \( C_0 = \dot{X}_0 X_0^{-1} \), where \( X_0 = X(0) \) and \( \dot{X}_0 = \dot{X}(0) \), is well-defined since \( \det(X_0) \neq 0 \). In the zero angular momentum case, as shown in \cite{TS2}, both eigenvalues of \( C_0 \) are real. We denote the spectrum of \( C_0 \) as:
\begin{equation}
\text{Spec}(C_0) = \{\lambda_1, \lambda_2\} \subset \mathbb{R} \,,
\end{equation}
and define
\begin{equation}
\pi_s = \min\{\lambda_1, \lambda_2\}\,.
\end{equation}

Finally, we suppose that during the motion, all mutual distances of the bodies are bounded above and below by the constants \(\alpha, \beta > 0\), with \(\alpha < \beta\), i.e.,
\begin{equation} \label{bounds}
\alpha \leq |z_{ij}(t)| \leq \beta, \quad  \forall \, i < j \,,
\end{equation}
and we define
\begin{equation} \label{ab}
\theta_{\alpha} = \frac{\sqrt{M}}{\alpha^{3/2}}, \quad \theta_{\beta} = \frac{\sqrt{M}}{\beta^{3/2}}, \quad M = m_1 + m_2 + m_3\,.
\end{equation}

\begin{theorem} \label{TH}
Let \( T_s > 0 \) be the moment in time when the first syzygy occurs. Then the following bounds hold:
\begin{equation} \label{main}
\frac{1}{\theta_{\alpha}} \mathrm{arccot}\left(- \frac{\pi_s}{\theta_{\alpha}}\right) \leq T_s \leq \frac{1}{\theta_{\beta}} \mathrm{arccot}\left(- \frac{\pi_s}{\theta_{\beta}}\right)\,.
\end{equation}
\end{theorem}

\begin{proof}
Let \( t \mapsto \phi_i(t) \), \( i = 1, 2, 3 \), be arbitrary positive smooth functions on \([0, +\infty)\). We consider a linear system of second order
\begin{equation} \label{phis}
\ddot{Z} = \mathcal{A} Z, \quad Z =
\begin{pmatrix}
Z_{11} & Z_{12} \\
Z_{21} & Z_{22}
\end{pmatrix}, \quad
\mathcal{A} (\phi_1,\phi_2,\phi_3) =  \begin{pmatrix}
-m_2 \phi_3 - m_{13} \phi_2 & m_1 \phi_{32} \\
m_2 \phi_{31} & -m_1 \phi_3 - m_{32} \phi_1
\end{pmatrix}\,,
\end{equation}
where \(\phi_{ij} = \phi_i - \phi_j\), $m_{ij}=m_i+m_j$.

If \((Z_{11}, Z_{12}) = (X_1, Y_1)\), \((Z_{21}, Z_{22}) = (X_2, Y_2)\), and \(\phi_i = \rho_i\) for \(i = 1, 2, 3\), then the equations \eqref{phis} and \eqref{EQ} coincide.

It is a straightforward computation to verify that for any solution \( Z \) of the system \eqref{phis}, the following identity holds:
\begin{equation}
\frac{1}{m_1} 
\begin{vmatrix}
Z_{11} & Z_{12} \\
\dot{Z}_{11} & \dot{Z}_{12}
\end{vmatrix}
+
\frac{1}{m_2}
\begin{vmatrix}
Z_{21} & Z_{22} \\
\dot{Z}_{21} & \dot{Z}_{22}
\end{vmatrix}
+
\frac{1}{m_3}
\begin{vmatrix}
Z_{11} + Z_{21} & Z_{12} + Z_{22} \\
\dot{Z}_{11} + \dot{Z}_{21} & \dot{Z}_{12} + \dot{Z}_{22}
\end{vmatrix}
= k, \quad k \in \mathbb{R}\,,
\end{equation}
which is an analog of the angular momentum conservation law in the three-body problem written in the form \eqref{EQ}.

The matrix \( B = \dot{Z} Z^{-1} \) is a solution of the following Riccati equation:
\begin{equation} \label{RIC}
\dot{B} + B^2 + R = 0, \quad R = -\mathcal{A}\,,
\end{equation}
an equation utilized in the study of the three-body problem in \cite{TS1}.

We introduce the matrices
\begin{equation}
I_2 = \left(
\begin{array}{rr}
1 & 0 \\
0 & 1
\end{array}
\right), \quad
\tilde{A}_1 = \left(
\begin{array}{rr}
\frac{m_{32}}{2} & 0 \\
-m_2 & -\frac{m_{32}}{2}
\end{array}
\right), \quad
\tilde{A}_2 = \left(
\begin{array}{rr}
-\frac{m_{13}}{2} & -m_1 \\
0 & \frac{m_{13}}{2}
\end{array}
\right)\,.
\end{equation}
Then, as shown in \cite[p. 6833]{TS2}, \( B \) can be presented, if \( k = 0 \), in the following form:
\begin{equation} \label{longer}
B = \frac{\dot{\delta}}{2\delta} I_2 + \frac{b}{m_2 \delta} \tilde{A}_1 - \frac{a}{m_1 \delta} \tilde{A}_2\,,
\end{equation}
where
\begin{equation} \label{alphab}
\beta = \frac{1}{2} \left( \frac{m_3}{m_1} + 1 \right), \quad \gamma = \frac{1}{2} \left( \frac{m_3}{m_2} + 1 \right) \,,
\end{equation}
and 
\begin{equation}
\delta =
\begin{vmatrix}
Z_{11} & Z_{12} \\
{Z}_{21} & {Z}_{22}
\end{vmatrix}, \quad 
a =
\begin{vmatrix}
Z_{11} & Z_{12} \\
\dot{Z}_{11} & \dot{Z}_{12}
\end{vmatrix}, \quad b =
\begin{vmatrix}
Z_{21} & Z_{22} \\
\dot{Z}_{21} & \dot{Z}_{22}
\end{vmatrix}\,.
\end{equation}

In order to apply Theorem \ref{EH1} to equation \eqref{RIC}, the solution \( B \) and the matrix \( R \)   should be symmetric. This can be achieved by the linear transformation \( \tilde{B} = P^{-1} B P \) with the invertible matrix \( P \) defined by:
\begin{equation}
P = \left( 
\begin{array}{cc}
-\frac{m_1}{m_{13}} & \frac{1}{m_{13}} \sqrt{\frac{m_1 m_3 M}{m_2}} \\
1 & 0 
\end{array}
\right), \quad M = m_1 + m_2 + m_3 \,.
\end{equation}

Indeed, it is easy to check that the matrices \( P^{-1} \tilde{A}_i P \), \( i=1,2 \), and \( \tilde{R} = P^{-1} R P \) are symmetric. Since \(\tilde{B}\) is a linear combination of \(I_2\) and \( P^{-1} \tilde{A}_i P \), it is also symmetric. Thus, equation \eqref{RIC} becomes:
\begin{equation}
\dot{\tilde{B}} + \tilde{B}^2 + \tilde{R} = 0\,,
\end{equation}
and Theorem \ref{EH1} can be applied.

\begin{lemma} \label{l1}
The matrix \( \tilde{\mathcal{A}}(\phi_1, \phi_2, \phi_3) = P^{-1} \mathcal{A}(\phi_1, \phi_2, \phi_3) P \) is negative semidefinite.
\end{lemma}

\begin{proof}
Since \( \tilde{\mathcal{A}} \) is symmetric, it is sufficient to show that the eigenvalues of \( \mathcal{A} \) are negative. We have:
\begin{equation}
\det(\mathcal{A}) = M (m_3 \phi_1 \phi_2 + m_2 \phi_1 \phi_3 + m_1 \phi_2 \phi_3) \geq 0, \quad 
\text{Tr}(\mathcal{A}) = - \left( m_{32} \phi_1 + m_{13} \phi_2 + m_{21} \phi_3 \right) \leq 0 \,,
\end{equation}
since \(\phi_i \geq 0\) for \(i = 1, 2, 3\).

Thus, \(\tilde{\mathcal{A}}\) is negative semidefinite.
\end{proof}

Let us consider two Riccati equations:
\begin{equation} \label{e1}
\dot{\tilde{B}}_1 + \tilde{B}_1^2 + \tilde{R}_1 = 0, \quad \tilde{R}_1 = -P^{-1} \mathcal{A}(\rho_1, \rho_2, \rho_3) P\,,
\end{equation}
and
\begin{equation} \label{e2}
\dot{\tilde{B}}_2 + \tilde{B}_2^2 + \tilde{R}_2 = 0, \quad \tilde{R}_2 = -P^{-1} \mathcal{A}\left( \frac{1}{\beta^3}, \frac{1}{\beta^3}, \frac{1}{\beta^3} \right) P \,
\end{equation}
subject to the same initial conditions
\begin{equation} \label{cond}
\tilde{B}_1(0) = \tilde{B}_2(0) = P^{-1}\dot{X}_0 X_0^{-1}P\,.
\end{equation}

By the linearity of \(\mathcal{A}\) as a function of \(\phi_i\) (\(i=1,2,3\)), we have
\[
\tilde{R}_2 - \tilde{R}_1 = P^{-1} \mathcal{A} \left( \rho_1 - \frac{1}{\beta^3}, \rho_2 - \frac{1}{\beta^3}, \rho_3 - \frac{1}{\beta^3} \right) P,
\]
and \(\rho_i - \frac{1}{\beta^3} \geq 0\) for \(i = 1, 2, 3\) by the definition of \(\beta\). 

According to Lemma \ref{l1}, it follows that \(\tilde{R}_1 \geq \tilde{R}_2\).

Equation \eqref{e2} can be easily solved by setting \(\tilde{R}_2 = \dot{Y} Y^{-1}\). This leads to the equivalent equation:
\begin{equation} \label{ee3}
\ddot{Y} = -\theta_{\beta}^2 Y, \quad \theta_{\beta} = \frac{\sqrt{M}}{\beta^{3/2}}\,.
\end{equation}
The solution to the Cauchy problem defined by \eqref{ee3} and the initial conditions $ Y(0)=Y_0=P^{-1} X_0 $, $\dot Y(0)=\dot{{Y}}_0=P^{-1} \dot{ X}_0  $  is:
\begin{equation} \label{eY}
Y(t) = \cos(\theta_{\beta} t)  Y_0 + \frac{1}{\theta_{\beta}} \sin(\theta_{\beta} t) \dot{{Y}}_0\,, \quad t \geq 0\,.
\end{equation}

Therefore, the corresponding maximal solution of \eqref{e2} is:
\begin{equation}
\tilde{B}_2(t) = \dot{Y}(t) Y(t)^{-1}, \quad t \in [0, t_2)\,,
\end{equation}
where \( t_2 \) is the first positive zero of \(t \mapsto \det(Y(t))\).

One has
\begin{equation}
\det(Y(t_2)) = 0 \, \Longleftrightarrow \, \det(\dot{X}_0 X_0^{-1} + \theta_{\beta} \cot(\theta_{\beta} t_2)I_2) = 0\,,
\end{equation}
since \( X_0 \) and  $P$  are  invertible matrices. As a consequence, we obtain:
\begin{equation}
- \theta_{\beta} \cot(\theta_{\beta} t_2) \in \text{Spec}(C_0), \quad C_0=\dot X_0 X_0^{-1}\,,
\end{equation}
and 
\begin{equation}
t_2 = \frac{1}{\theta_{\beta}} \mathrm{arccot}\left(- \frac{\pi_s}{\theta_{\beta}}\right)\,.
\end{equation}

Let \( T_s = t_1 > 0 \) be the first syzygy moment for the solution of the three-body problem \( t \mapsto z_i(t) \), \( i = 1, 2, 3 \). Then, using a similar argument, we show that the maximal solution of the Cauchy problem defined by \eqref{e1} and \eqref{cond} is defined in the interval \([0, t_1)\). Therefore, according to Theorem \ref{EH1}, \( T_s \leq t_2 \), and the upper bound in \eqref{main} is proven.

To prove the lower bound, consider the Riccati equations:
\begin{equation}
\dot{\tilde{B}}_1 + \tilde{B}_1^2 + \tilde{R}_1 = 0, \quad \tilde{R}_1 = -P^{-1} \mathcal{A}\left( \frac{1}{\alpha^3}, \frac{1}{\alpha^3}, \frac{1}{\alpha^3} \right) P \,,
\end{equation}
and
\begin{equation}
\dot{\tilde{B}}_2 + \tilde{B}_2^2 + \tilde{R}_2 = 0, \quad \tilde{R}_2 = -P^{-1} \mathcal{A}(\rho_1, \rho_2, \rho_3) P\,,
\end{equation}
where \(\tilde{R}_1 \geq \tilde{R}_2\). Applying the same arguments as before, the proof of Theorem \ref{TH} is complete.
\end{proof}

\begin{remark}
The uniform bounds \( |z_{ij}| \geq \alpha \), $\forall \, i<j$  have a very natural astrophysical interpretation: such motion corresponds to the collision-free movement of three rigid planets, each having the same radius \( R = \alpha / 2 \).
As seen from the proof of Theorem \ref{TH}, the uniform bounds \eqref{bounds} and the absence of collisions are required only for the period of time preceding the first syzygy.
\end{remark}

\section{Conclusion: Numerical Validation and Open Questions}

To numerically verify our findings, we will consider the figure-eight periodic solution with equal masses \(m_1 = m_2 = m_3=1\) and zero angular momentum, as described in \cite{A3}. The initial positions and velocities of the bodies are set as follows:
\begin{equation} \label{1}
z_1(0) = 1.08075 - i 0.0126893, \quad z_2(0) = -0.570154 + i 0.350807, \quad z_3(0) = -z_1(0) - z_2(0),
\end{equation}
and
\begin{equation} \label{2}
\dot{z}_1(0) = 0.0193421 + i 0.467219, \quad \dot{z}_2(0) = 1.0852 - i 0.174718, \quad \dot{z}_3(0) = -\dot{z}_1(0) - \dot{z}_2(0)\,.
\end{equation}

Numerically, we can find that the first syzygy occurs at \( T_s = 0.55431 \) (see Figure \ref{PIC}). The constants \(\alpha\) and \(\beta\) are estimated to be \(\alpha = 0.690526\) and \(\beta = 2\). Using formulas \eqref{ab}, we determine \(\theta_{\alpha} = 3.01849\) and \(\theta_{\beta} = 0.612372\).  

Using the initial conditions \eqref{1} and \eqref{2}, we compute:
\begin{equation}
X_0 = \begin{pmatrix}
0.734528 & 1.35841 \\
0.755791 & -0.470708
\end{pmatrix}, \quad \text{Spec}(X_0) = \{1.31082, -1.047\}, \quad \pi_s = -1.047,
\end{equation}
and the bounds \eqref{main} give us:
\begin{equation}
0.409781 \leq T_s \leq 0.864231,
\end{equation}
which are quite satisfactory.

We would like to highlight several open questions. It would be interesting to find bounds analogous to \eqref{main} for the velocity syzygy moments using similar ideas based on the comparison of matrix Riccati equations. It is important to note that the estimates provided by our Theorem \ref{TH} are valid only if the initial configuration at \(t = 0\) is not a syzygy. However, with some effort and by employing Theorem \ref{EH1}, one could establish bounds similar to \eqref{main} even in the case where the initial configuration is a syzygy.

Finally, generalising to the case of non-zero angular momentum is an intriguing and important challenge.

\vspace{0.5cm}

\noindent {\bf Akcnowledgments}\\

\noindent I would like to express my special gratitude  to  Richard Montgomery for   useful suggestions and valuable remarks and to Jean-Claude Sikorav for attracting my attention to work \cite{EH}.

\clearpage
\begin{figure}[tbp] 
\centering
\includegraphics[width=0.6\linewidth]{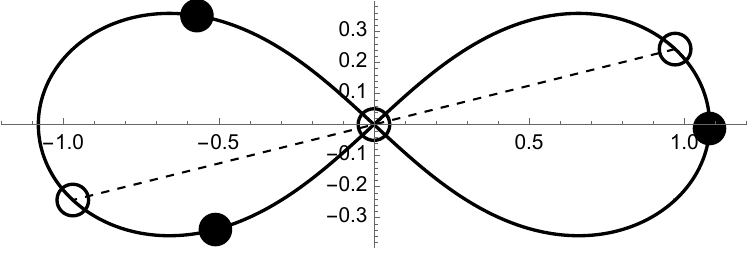}
\caption{The figure shows the trajectory of the figure-eight solution, as described in \cite{A3}. The initial positions of the bodies at \( t = 0 \) are represented by filled points. Additionally, the configuration of the first syzygy, which occurs at \( t = 0.55431 \), is indicated by unfilled points.}
\label{PIC}
\end{figure}

\end{document}